\documentclass[12pt]{amsart}
\usepackage{geometry}
\usepackage[mathscr]{euscript}
\usepackage{ amssymb }
\usepackage[comma, numbers]{natbib}
\usepackage{bm}
\usepackage{enumerate}
\usepackage[english]{babel}
\usepackage{commath}
\usepackage{tikz-cd}
\usepackage{graphicx}
\usepackage [autostyle, english = american]{csquotes}
\MakeOuterQuote{"}

\theoremstyle{plain}
\newtheorem{theorem}{Theorem}[section]

\theoremstyle{plain}
\newtheorem{lemma}[theorem]{Lemma}

\theoremstyle{plain}

\theoremstyle{definition}
\newtheorem{definition}[theorem]{Definition}

\theoremstyle{plain}
\newtheorem{proposition}[theorem]{Proposition}

\theoremstyle{remark}

\theoremstyle{definition}
\newtheorem{example}[theorem]{Example}

\theoremstyle{plain}

\theoremstyle{plain}

\theoremstyle{plain}

\title{Some counterexamples in abstract coarse geometry}
\author{Pawel Grzegrzolka}
\address{University of Tennessee, Knoxville, USA}
\email{pgrzegrz@vols.utk.edu}

\author{Jeremy Siegert}
\address{University of Tennessee, Knoxville, USA} 
\email{jsiegert@vols.utk.edu}
\date{\today} 


\begin{document}

\begin{abstract}
We construct an example of a coarse proximity space that is not induced by any coarse structure.  We then show how to "stitch" two coarse proximity spaces with homeomorphic boundaries into one coarse proximity space. Finally, we construct a coarse proximity space that is not induced by any asymptotic resemblance structure.
\end{abstract}

\maketitle

\section{Introduction}
The relationship between uniform structures and proximity structures is very well-understood (see for example chapter $3$ in \cite{proximityspaces}). In particular, it is known that every proximity structure is induced by some uniformity. In this paper, we show that the large-scale equivalents of uniform structures and proximity structures, namely coarse structures and coarse proximity structures, do not exhibit the same behavior. In particular, we construct an example of a coarse proximity space that is not induced (as in \cite{paper2}) by any coarse structure.  Then, we show how to "stitch" two coarse proximity spaces with homeomorphic boundaries into one coarse proximity space. Finally, we use the "stitching" to construct a coarse proximity space that is not induced (as in \cite{paper2}) by any asymptotic resemblance structure.

All necessary notation and background needed to understand the examples in this paper can be found in \cite{paper2}. An interested reader can find more information about coarse 
proximities in \cite{paper1} and \cite{paper3}, about coarse spaces in \cite{Roe}, and about asymptotic resemblance spaces in \cite{Honari}.

\section{Coarse proximity not induced by any coarse structure}
In this section, we construct a coarse proximity space that is not induced by any coarse structure. This construction has been suggested by Thomas Weighill. 

\begin{definition}
Let $X$ be a set with a bornology $\mathcal{B}.$ Then $\mathcal{A}\subseteq \mathcal{B}$ is called a \textbf{basis} for a bornology $\mathcal{B}$ if every $B \in \mathcal{B}$ is contained in some $A \in \mathcal{A}.$
\end{definition}

\begin{lemma}\label{needed lemma}
	Let $(X,\mathcal{E})$ be a coarse space which induces a coarse proximity structure $(X,\mathcal{B},{\bf b})$. If $\mathcal{B}$ has a countable basis, then for all $A,B\subseteq X$ we have that $A{\bf b}B$ if and only if there are unbounded countable sets $A_{0}\subseteq A$ and $B_{0}\subseteq B_{0}$ such that $A_{0}{\bf b}B_{0}$.
\end{lemma}
\begin{proof}
	The statement is clearly true when either $A$ or $B$ is bounded, so let us assume that both $A$ and $B$ are unbounded. The backward direction is clear. To see the forward direction, assume $A{\bf b}B.$ By definition, this implies that there is an entourage $E\in\mathcal{E}$ such that for all $D\in\mathcal{B}$ we have 
	\[[(A\setminus D)\times(B\setminus D)]\cap E\neq\emptyset.\]
	Let $\{D_{n}\}_{n\in\mathbb{N}}$ be a countable basis for $\mathcal{B}$. Without loss of generality we may assume that $D_{i}\subseteq D_{i+1}$ for each $i$. Then for each $n\in\mathcal{B},$ let $(x_{n},y_{n})\in A\times B$ be chosen so that
	\[(x_{n},y_{n})\in[(A\setminus D_{n})\times(B\setminus D_{n})]\cap E.\]	
	Define $A_{0}:=\{x_{n}\}_{n\in\mathbb{N}}$ and $B_{0}:=\{y_{n}\}_{n\in\mathbb{N}}$. Then $A_{0}$ add $B_{0}$ are countable, unbounded, and we trivially have that $A_{0}{\bf b}B_{0}$.
\end{proof}

Now we will construct a coarse proximity space $(Y, \mathcal{B}, {\bf b})$ (whose bornology has a countable basis) and two unbounded subsets $A$ and $B$ such that $A{\bf b}B$ but for all unbounded countable sets $A_0 \subseteq A$ and $B_0 \subseteq B,$ we have that $A_0 \bar{\bf b}B_0.$ By the above lemma, this will show that the coarse proximity structure on $Y$ is not induced by any coarse structure on $Y.$

Let $I:=[0,1]$ be the closed interval with its metric topology and define $X:=I^{I}$ to be the corresponding product space with its product topology. By Tychonoff's theorem, this space is compact. Let $(Y,\mathcal{B},{\bf b})$ be the coarse proximity space defined by $Y:=X\times[0,1)$, $\mathcal{B}$ is all precompact sets in $Y$, and for any subsets $A,B\subseteq Y$ define $A{\bf b}B$ if and only if their closures in $X\times[0,1]$ intersect in $X\times\{1\}$. The fact that $(Y,\mathcal{B},{\bf b})$ is indeed a coarse proximity space follows from Proposition 6.1 in \cite{paper3}. Define
\[A:=\{(0)_{t\in I}\}\times[0,1).\]
To define $B,$ recall that for any nontrivial interval $J$ (open, closed, or half open) and any positive natural number $k$ we have that
\[|I|=|J|=|\mathbb{R}|=|\mathbb{R}^{k}|=|J^{k}|,\]
where $|\cdot|$ denotes the cardinality of a set. Thus, it is easy to see that for any such interval $J$ and any positive natural number $k$ we have that there is a bijection from $J$ to the set of subsets of $I$ of size $k.$ To construct $B,$ define $\phi_{1}$ to be any bijection between $[0,1/2]$ and the set of singletons in $[0,1]$. Define $\phi_{2}$ to be any bijection between $(1/2,3/4]$ and the subsets of $[0,1]$ of size $2$. In general, for any $n \geq 2,$ define $\phi_{n}$ to be a bijection from $(\frac{2^{n-1}-1}{2^{n-1}}, \frac{2^n-1}{2^n}]$ and the subsets of $[0,1]$ of size $n$. This gives an injective function $\phi$ from $[0,1)$ to finite subsets of $I$. We then define

\[B:=\left\{\{(x_{t})\}\times\{k\}\in Y\middle| k \in [0,1),  \begin{array}{ll}
x_{t}=0 &\text{if} \quad t\in\phi(k)\\
x_{t}=1 &\text{if} \quad t\notin\phi(k)
\end{array}\right\}.\]

To see that $A{\bf b}B,$ notice that the intersection of the closure of $A$ in $X \times [0,1]$ and $X\times\{1\}$ is precisely the point $\{(0_{t})\}\times\{1\}$. Denoting the intersection of the closure of $B$ in $X \times [0,1]$ and $X\times\{1\}$ by $B^{\prime},$ we claim that $\{(0_{t})\}\times\{1\}\in B^{\prime}$. Let $U$ be a basic open set in $X \times [0,1]$ containing $\{(0_{t})\}\times\{1\}$. Say $U=V\times W,$ where $W\subseteq[0,1]$ is an open set containing $1$. More specifically, we can assume that $W=(1-\epsilon,1]$. Let $t_{1},t_{2},\ldots,t_{n}$ be finitely many elements of $[0,1]$ such that $\pi_{t}(V)\neq[0,1]$ if and only if $t\in\{t_{1},\ldots,t_{n}\}$. Chose $m\in\mathbb{N}$ such that $m>n$ and $\frac{2^{m-1}-1}{2^{m-1}}>1-\epsilon.$ By the definition of $\phi,$ there is some $k\in(\frac{2^{m-1}-1}{2^{m-1}}, \frac{2^m-1}{2^m}]$ such that $\{t_{1},\ldots,t_{n}\}\subseteq \phi(k)$. Then the point $\{(x_{t})\}\times\{k\}\in B$ characterized by 
\[x_{t}=\left\{\begin{array}{ll}
0 & t\in\phi(k)\\
1 & t\notin\phi(k)\}
\end{array}\right\}\]
is an element of $B$ contained in $U.$ Therefore, $\{(0_{t})\}\times\{1\} \in B',$ and consequently $A{\bf b}B.$ 

Now let $A_{0}\subseteq A$ and $B_{0}\subseteq B$ be unbounded countable subsets. Note that the intersection of the closure of $A_0$ in $X\times[0,1]$ with $X\times\{1\}$ is precisely the singleton $\{(0_{t})\}\times\{1\}.$ Consider $B_{0}.$ We know that for each $\{(x_{t})\}\times\{k\} \in B_0,$ only finitely many $x_{t}$ are equal to $0$ and the rest of them are equal to $1.$ Let $C\subseteq[0,1]$ be all $t$ such that there is some $\{(x_{s})\}\times\{k\} \in B_0$ for which $x_{t}=0$. In other words, $C$ is the collection of all those "coordinates" of $X$ where at least one element of $B_0$ has value $0.$ Because $B_{0}$ is countable and each point of $B_0$ is such that only finitely many coordinates are equal to $0,$ we have that $C$ is countable. Let $t_{0}$ be an element of $[0,1]\setminus C.$ Define an open neighborhood $U$ of $\{(0_{t})\}\times\{1\}$ by
\[U:=\left(\prod_{t\in[0,1]}U_{t}\right)\times[0,1],\]
where $U_{t_0}:=[0,1/3)$ and $U_{t}=[0,1]$ for all $t\in[0,1]$ for which $t\neq t_0$. Then $U$ is a well-defined open neighborhood of $\{(0_{t})\}\times\{1\}$ in $X\times[0,1]$. However, note that for all $\{(x_{t})\}\times\{k\}\in B_{0},$ we have that $x_{t_{0}}=1$. Therefore, $B_{0}\cap U=\emptyset,$ i.e., $\{(0_{t})\}\times\{1\}$ is not in the closure of $B_0$ in $X \times [0,1].$ Consequently, we have that $A_{0}\bar{\bf b}B_{0}.$

It is left to show that the bornology on $Y$ has a countable basis. Let $\mathcal{A}:=\{X\times[0,r] \subseteq X \times [0,1) \mid r \text{ rational}\}.$ Clearly $\mathcal{A}$ is countable and each set $X\times[0,r_{n}]$ is an element of $\mathcal{B}.$ Also, each element of $\mathcal{B}$ is a subset of $X\times[0,r]$ for some $r.$ Thus, $\mathcal{A}$ is a countable basis for $\mathcal{B}.$

\section{Boundary stitching of coarse proximity spaces}

Now we will construct a coarse proximity space by "stitching" two coarse proximity spaces at their homeomorphic boundaries. Let $(X,\mathcal{B}_{X},{\bf b}_{X})$ and $(Y,\mathcal{B}_{Y},{\bf b}_{Y})$ be coarse proximity spaces for which the boundaries $\mathcal{U}X$ and $\mathcal{U}Y$ are homeomorphic. Let $\mathcal{B}_{X\sqcup Y}$ be the collection of all subsets of the disjoint union $X\sqcup Y$ that are subsets of some finite union of sets in $\mathcal{B}_{X}$ and $\mathcal{B}_{Y}$. Then $\mathcal{B}_{X\sqcup Y}$ is a bornology on $X\sqcup Y$. Let $g:\mathcal{U}Y\rightarrow\mathcal{U}X$ be a homeomorphism. If $A\subseteq X\sqcup Y$ is unbounded in $(X\sqcup Y,\mathcal{B}_{X\sqcup Y})$ then one of $A\cap X$ or $A\cap Y$ is unbounded. We will denote the trace of $A \cap X$ on $\mathcal{U}X$ by $A_{X}^{\prime}$ and the trace of $A \cap Y$ on $\mathcal{U}Y$ by $A_{Y}^{\prime}$. Then for unbounded $A,B\subseteq X\sqcup Y,$ we define
\[A{\bf b}_{g}B\iff [A_{X}^{\prime}\cup g(A_{Y}^{\prime})]\cap[B_{X}^{\prime}\cup g(B_{Y}^{\prime})]\neq\emptyset.\]
\begin{proposition}
With the notation above, the following are true:
\begin{enumerate}
\item $(X\sqcup Y,\mathcal{B}_{X\sqcup Y},{\bf b}_{g})$ is a coarse proximity space,
\item $\mathcal{U}(X\sqcup Y)$ is homeomorphic to $\mathcal{U}X.$
\end{enumerate}
\end{proposition}
\begin{proof}
To see (1), let us show that the the strong axiom holds (since the other axioms are evident). Let $A\overline{\bf b}_gB.$ This means that
\[[A_{X}^{\prime}\cup g(A_{Y}^{\prime})]\cap[B_{X}^{\prime}\cup g(B_{Y}^{\prime})]=\emptyset.\]
To find the desired set $E,$ consider $\mathfrak{X}\sqcup\mathfrak{Y}$ with the disjoint union topology. Create the set $\mathfrak{X}\sqcup_g\mathfrak{Y}$ by glueing $\mathfrak{X}$ and $\mathfrak{Y}$ along $\mathcal{U}X$ and $\mathcal{U}Y$ using $g.$ Equip that glued space with the usual quotient topology coming from the projection map $\pi: \mathfrak{X}\sqcup\mathfrak{Y} \to \mathfrak{X}\sqcup_g\mathfrak{Y}.$ It is easy to see that $\mathfrak{X}\sqcup_g\mathfrak{Y}$ is compact and Hausdorff, and consequently normal. Notice that $A_{X}^{\prime}\cup g(A_{Y}^{\prime})$ and $B_{X}^{\prime}\cup g(B_{Y}^{\prime})$ are closed and disjoint subsets of $\mathfrak{X}\sqcup_g\mathfrak{Y}.$ By Urysohn's Lemma, there exists a continuous map $f:\mathfrak{X}\sqcup_g\mathfrak{Y} \to [0,1]$ such that $f(x)=0$ for all $x \in (A_{X}^{\prime}\cup g(A_{Y}^{\prime}))$ and $f(x)=1$ for all $x \in (B_{X}^{\prime}\cup g(B_{Y}^{\prime})).$ Let 
\[E:= \pi^{-1}(f^{-1}((\frac{2}{3}, 1])) \cap (X \sqcup Y).\]
Notice that $E$ is an open set in $(X \sqcup Y).$ Also, notice that $E'_X$ (i.e., the trace of $E \cap X$ in $\mathcal{U}X$) does not intersect $A_X' \cup g(A'_Y)$ (because otherwise $f$ has to take the value of $0$ and some other value in $[\frac{2}{3},1]$ at a point of intersection). Similarly, $g(E'_Y)$ does not intersect $A_X' \cup g(A'_Y).$ In other words,
\[[A_{X}^{\prime}\cup g(A_{Y}^{\prime})]\cap[E_{X}^{\prime}\cup g(E_{Y}^{\prime})]=\emptyset,\]
i.e.,  $A\overline{\bf b}_g E.$ Similarly one can easily show that $B\overline{\bf b}_g ((X \sqcup Y) \setminus E).$ This finishes the proof of the strong axiom.

To show (2), let $\delta_g$ be the discrete extension of ${\bf b}_g.$ Notice that the inclusion map $h: X \to (X \sqcup Y)$ is a coarse proximity map. Consequently, by Corollary 4.12 in \cite{paper3}, the unique extension of $h$ between the Smirnov compactifications of $X$ and $(X \sqcup Y)$ is continuous and maps $\mathcal{U}X$ to $\mathcal{U}(X \sqcup Y).$ Denote that map from $\mathcal{U}X$ to $\mathcal{U}(X \sqcup Y)$ by $\bar{h}.$ To show that this map is injective, let $\sigma_1, \sigma_2 \in \mathcal{U}X$ be clusters with identical image $\sigma_3.$ Recall that by Proposition 3.5 in \cite{paper3}, $ \sigma_3$ is given by
\[\sigma_3=\{A \subseteq (A \sqcup B) \mid \forall \, B \in \sigma_1, A \delta_g B\}=\{A \subseteq (A \sqcup B) \mid \forall \, B \in \sigma_2, A \delta_g B\}.\]
Notice that $\sigma_1 \subseteq \sigma_3.$ For if $A \in \sigma_1,$ then $A{\bf b}_X B$ for all $B \in \sigma_1,$ which is equivalent to $A_X' \cap B_X' \neq \emptyset$ for all $B \in \sigma_1.$ This in turn implies that $A{\bf b}_g B$ for all $B \in \sigma_1.$ Thus, $A \in \sigma_3.$ Similarly one can show that $\sigma_2 \subseteq \sigma_3.$ To see that $\sigma_1 \subseteq \sigma_2,$ let $A \in \sigma_1.$ Since $\sigma_1 \subseteq \sigma_3,$ $A\delta_g B$ for all $B \in \sigma_2.$ Since both $A$ and $B$ are subsets of $X,$ this means that $A$ and $B$ intersect or $A_X' \cap B_X' \neq \emptyset$ (since $g(A_Y')=g(B_Y)'= \emptyset$). In particular, this implies that $A \cap B \neq \emptyset$ or $A{\bf b}_X B$ for all $B \in \sigma_2.$ Thus, $A \in \sigma_2.$ Thus, $\sigma_1 \subseteq \sigma_2.$ The opposite inclusion follows similarly. Thus, $\bar{h}$ is injective on $\mathcal{U}X.$ Finally, to see that $\bar{h}$ is surjective, let $\sigma$ be an arbitrary element of $\mathcal{U}(X \sqcup Y).$ Consider
\[\sigma':= \{A \subseteq X \mid A\in \sigma\}.\]
To see that $\sigma$ is a cluster in $\mathcal{U}X,$ let $\delta_X$ denote the discrete extension of ${\bf b}_X$ and let $A,B \in \sigma'.$ Then $A,B \in \sigma$ and thus $A \delta_b B.$ Since both $A$ and $B$ are strictly subsets of $X$ (and not $Y$), this implies that $A_X' \cap B_X' \neq \emptyset,$ which in turn implies that $A{\bf b}_XB.$ Consequently, $A\delta_X B.$ To see the second axiom of a cluster, let $A \delta_X C$ for all $C \in \sigma'.$ Consequently, $A \cap C \neq \emptyset$ or $A {\bf b}_X C$ for all $C \in \sigma'.$ W claim that this implies that $A \cap C \neq \emptyset$ or $A {\bf b}_g C$ for all $C \in \sigma.$ If that is not the case, then there exists $C \in \sigma$ such that $A \cap C = \emptyset$ and $A \bar{\bf b}_g C.$ Since $C \in X \sqcup Y,$ $C= D \cup E$ where $D \subseteq X, E \subseteq Y,$ and $D \cap E = \emptyset.$ Since $C \in \sigma,$ either $D \in \sigma$ or $E \in \sigma.$ If $D \in \sigma,$ then $D \in \sigma'$ and we have a contradiction, since $D \cap A = \emptyset$ and $D \bar{\bf b}_X A,$ i.e., $D \bar{\delta}_X A.$ Thus, it has to be that $E \in \sigma.$ Since $cl_{\mathfrak{X}}(A)$ and $g(E_Y')$ are closed and disjoint in $\mathfrak{X},$ by the application of Urysohn's Lemma in $\mathfrak{X}$ on $cl_{\mathfrak{X}}(A)$  and $g(E_Y')$, we get an open set $F$ in $X$ that is disjoint from $A$ and whose trace in $\mathcal{U}X$ contains $g(E_Y')$ while at the same time is disjoint from the trace of $A.$ Notice that $F \in \sigma.$ It is because if $H$ is an arbitrary element of $\sigma,$ then $H{\bf b}_g E$ (since they both belong to $\sigma$). Consequently,
\[(H_X' \cup g(H_Y'))\cap g(E_Y') \neq \emptyset,\]
which implies that 
\[(H_X' \cup g(H_Y'))\cap F_X' \neq \emptyset,\]
which implies that $F\delta_gH.$ Since $H$ was an arbitrary element of $\sigma,$ this shows that $F \in \sigma.$ But this is a contradiction, since $F \in \sigma',$ but $F\bar{\delta}_XA.$ Thus, it has to be that $A \cap C \neq \emptyset$ or $A {\bf b}_g C$ for all $C \in \sigma.$ This shows that $A \in \sigma$ and consequently $A \in \sigma'.$ Finally, to prove the last axiom of a cluster, let $A \cup B \in \sigma'.$ Then $A \cup B \in \sigma.$ Consequently, $A \in \sigma$ or $B \in \sigma,$ which shows that either $A \in \sigma'$ or $B \in \sigma'.$ This finishes the proof that $\sigma'$ is a cluster in $\mathcal{U}X$ (because it is clear that $\sigma'$ can only contain unbounded sets). Finally, let us show that the image of $\sigma'$ under $g$ is $\sigma.$ Recall that 
\[g(\sigma')=\{A \subseteq (A \sqcup B) \mid \forall \, B \in \sigma', A \delta_g B\}.\]
If $A \in \sigma,$ then clearly $A \delta_g B$ for all $B \in \sigma'$ (because any $B \in \sigma'$ is also an element of $\sigma$). Consequently, $A \in g(\sigma')$. Therefore, $\sigma \subseteq g(\sigma'),$ which implies that $\sigma = g(\sigma').$ This finishes the proof that $\bar{g}$ is a continuous bijective function between compact and Hausdorff spaces $\mathcal{U}X$ and $\mathcal{U} (X \sqcup Y),$ i.e., $\bar{g}$ is a homeomorphism from $\mathcal{U}X$ to $\mathcal{U} (X \sqcup Y).$
\end{proof}
\section{Coarse proximity not induced by any asymptotic resemblance}
In this section, we construct a coarse proximity space that is not induced by
any asymptotic resemblance structure.

\begin{lemma}\label{stronk implies weak}
	Let $(X,\lambda)$ be an asymptotic resemblance space which induces a coarse proximity structure $(X,\mathcal{B},{\bf b})$ with corresponding weak asymptotic resemblance relation $\phi$. If $A,B\subseteq X$ are unbounded sets such that $A\lambda B$, then $A\phi B$.
\end{lemma}
\begin{proof}
Let $A,B\subseteq X$ be unbounded sets such that $A\lambda B$ and let $B_{0}\subseteq B$ be an unbounded set. We wish to show that $B_{0}{\bf b}A$. If $B_{0}=B,$ then we are done. Otherwise, we have that $(B_{0}\cup(B\setminus B_{0}))\lambda A$ which implies that $A=(A_{1}\cup A_{2}),$ where $A_{1}$ and $A_{2}$ are nonempty and are such that $A_{1}\lambda B_{0}$ and $A_{2}\lambda (B\setminus B_{0})$. Because $B_{0}$ is unbounded we have that $A_{1}$ is also unbounded. Then $A_{1}\subseteq A$ is such that $A_{1}\lambda B_{0},$ and hence $A{\bf b}B_{0}$. Similarly, if $A_{0}\subseteq A$ is unbounded, then $A_{0}{\bf b}B$. Thus, $A\phi B$. 
\end{proof}

\begin{proposition}\label{needed corollary}
	Let $(X,\lambda)$ be an asymptotic resemblance space that induces the coarse proximity structure $(X,\mathcal{B},{\bf b})$ with corresponding weak asymptotic resemblance $\phi$. If $A_{1}$ and $B_{1}$ are unbounded sets such that $A_{1}{\bf b}B_{1},$ then there are unbounded sets $A_{2}\subseteq A_{1}$ and $B_{2}\subseteq B_{1}$ such that $A_{2}\phi B_{2}$.
\end{proposition}
\begin{proof}
Immediate consequence of Lemma \ref{stronk implies weak}.
\end{proof}

\begin{example}
Let $(X,d)$ be an unbounded proper metric space with metric coarse proximity structure $(X,\mathcal{B}_{1},{\bf b}_{1})$. Then $\mathcal{U}X$ is homeomorphic to the Higson corona $\nu X$. Let $(Y,\mathcal{B}_{2},{\bf b}_{2})$ be the "box space" (i.e., $Y:= \nu X \times [0,1),$ see Section 6 of \cite{paper3}) corresponding to $\nu X$. Let $g:\nu X\rightarrow \nu X\times\{1\}$ be the homeomorphism defined by $x\mapsto (x,1)$. Consider the coarse proximity space $(X\sqcup Y,\mathcal{B}_{X\sqcup Y},{\bf b}_{g})$. Recall that for any $x\in\nu X,$ there is no unbounded set $A\subseteq X$ whose trace on the Higson corona $\nu X$ is precisely $\{x\}$. Let $x_{0}\in\nu X$ be given. Define $A_{1}\subseteq Y$ to be $\{x_{0}\}\times[0,1)$ and define $B_{1}=X$. Then $A_{1}{\bf b}_{g}B_{1}$, but if $A_{2}\subseteq A_{1}$ and $B_{2}\subseteq B_{1}$ are any two unbounded sets, then we have that $A_{2}\bar{\phi} B_{2}$ because $\phi$ related sets have the same trace, but  $A_{2}^{\prime}=\{x_{0}\}$ and $B_{2}^{\prime}$ cannot be precisely $\{x_0\}.$ Therefore, the coarse proximity space $(X\sqcup Y,\mathcal{B}_{X\sqcup Y},{\bf b}_{g})$ is not induced by any asymptotic resemblance by Proposition \ref{needed corollary}.
\end{example}

\bibliographystyle{abbrv}
\bibliography{counterexamples}{}

\end{document}